\documentclass[12pt,reqno]{amsart}
\usepackage{cite}
\usepackage{amssymb,latexsym,amsmath,amsfonts}
\usepackage{latexsym}
\usepackage[mathscr]{eucal}
= 8.9in
\theoremstyle{plain}
\newtheorem{theorem}{\textbf{Theorem}}[section]
\newtheorem{lemma}[theorem]{\textbf{Lemma}}
\newtheorem{remark}[theorem]{\textbf{Remark}}

\newtheorem{definition}[theorem]{Definition}
\newtheorem*{Acknowledgement}{\textnormal{\textbf{Acknowledgement}}}
\newtheorem{example}[theorem]{\textbf{Example}}

\begin{document}

%%%%% DO NOT CHANGE THE FOLLOWING LINES %%%%%%%%%%%
%%%%%%%%%%%%%%%%%%%%%%%%%%%%%%%%%%%%%%%%%%%%%%%%%%%

% PLEASE KEEP ALL IMAGES IN THE SAME DIRECTORY AS THE TEX FILE AND PLEASE
% KEEP ALL WORK WITHIN ONE TEX FILE. LATEX HAS NOW MOVED AWAY FROM DVI
% AND USES PDF INSTEAD, EPS FIGURES ARE NOT SUPPORTED IN PDFLATEX SO PLEASE
% DO NOT USE EPS OR PS FIGURES, CONVERT THEM TO PDF INSTEAD.
% THIS CAN BE ACHIEVED USING PS2PDF OR GHOSTSCRIPT FOR EXAMPLE.
% WHEN USING PS2PDF, THE OPTION -dEPSCrop PRESERVES THE BOUNDING BOX, E.G.
% PS2PDF -dEPSCrop FIGURE1.EPS

% PLEASE LEAVE THE TITLE CODE ON ONE LINE.
% PLEASE DO NOT ADD ANYTHING BEFORE \title ON THE SAME LINE.
% THE FIRST TITLE (IN []) IS THE SHORT TITLE USED FOR PAGE HEADERS (YOU MAY
% NEED TO ABBREVIATE YOUR TITLE SO THAT IT FITS IN THE PAGE HEADER) THE
% SECOND IS THE ARTICLE TITLE THAT APPEARS AT THE BEGINNING OF THE PAPER.
\title[Fixed Point Results On $\theta$-metric Spaces via Simulation Functions]{Fixed Point Results On $\theta$-metric Spaces via Simulation Functions}

% INCLUDE AUTHOR INFORMATION IN THE FOLLOWING LINES
% PLEASE USE A DIFFERENT \author FOR EACH AUTHOR, LEAVING THEM ON SEPARATE
% LINES. PLEASE DO NOT ADD ANYTHING BEFORE \author ON THE SAME LINE.
% IN THE EXAMPLE BELOW M. SMITH AND F. SMITH ARE FROM THE SAME DEPARTMENT
% BUT A. R. SMITH IS FROM A SECOND INSTITUTE.
% THE NAMES IN []s ARE DISPLAYED AS HEADERS ON EACH PAGE (AND SHOULD NOT INCLUDE
% FOOTNOTES TO INSTITUTES), AND THOSE IN {}s ARE AT THE TOP OF THE ARTICLE.
% USE A '\' AFTER PERIODS AS THIS STOPS LATEX THINKING THAT THEY MARK THE END
% OF A SENTENCE.
\author[Ankush \ Chanda]{Ankush \ Chanda$^{1}$}
\address{$^1$Department of Mathematics,
\newline \indent National Institute of Technology  Durgapur,
\newline \indent India}
\email{ankushchanda8@gmail.com}

\author[Lakshmi Kanta \ Dey]{Lakshmi Kanta \ Dey$^{2}$}
\address{$^{2}$Department of Mathematics,
\newline \indent National Institute of Technology  Durgapur,
\newline \indent India}
\email{lakshmikdey@yahoo.co.in}

% INSERT KEYWORDS AND SUBJECT CLASSIFICATIONS BELOW, PLEASE DO NOT CHANGE THE FORMATTING AND
% IF YOU ARE ONLY USING PRIMARY CLASSIFICATIONS LEAVE THE PRIMARY KEYWORD (THIS IS USED FOR INDEXING)
\keywords{$\theta$-metric space, simulation functions, $\mathcal{Z}$-contraction, modified $\mathcal{Z}-$contraction.
% THE FOLLOWING WILL BE ADDED BY THE EDITOR (DO NOT REMOVE THE FOLLOWING LINES)
%\indent {\it Received}: \\
%\indent {\it Accepted}:
}

\begin{abstract}
In a recent article, Khojasteh et al. introduced a new class of simulation functions, $\mathcal{Z}$-contractions, with blending over known contractive conditions in the literature. Subsequently, in this paper, we extend and generalize the results on $\theta$-metric context and we discuss some fixed point results in connection with existing ones. Also, we originate the notion of modified $\mathcal{Z}-$contractions and explore the existence and uniqueness of fixed points of such functions on the said spaces. Finally we include examples to instantiate our main results.
\end{abstract}

\maketitle

\section{Introduction}
\baselineskip .55 cm
With extensive and manifold applications, fixed point theory has been one of the most influential research topics in various field of engineering and science. The most incredible result in this direction was stated by Banach, known as the Banach contraction principle \cite{ba}. This remarkable result has been generalized and extended in various abstract spaces using different conditions. However, the prospect of fixed point theory charmed many researchers and so there is a vast literature available for readers \cite{Ch,De,DD,DM,Rk1,GRRS}.%cite{},\cite{}\cite{}\cite\cite{GRRS}.

One of the most impressive generalizations of the notion of a metric is the concept of a fuzzy metric. Motivated from the definition of fuzzy metric spaces, recently Khojasteh et al. \cite{KKR} introduced $\theta-$metric by replacing the triangle inequality with a more generalized inequality.

In recent times, Khojasteh et al. \cite{KSR} introduced the notion of $\mathcal{Z}$-contraction by using a new class of auxiliary functions called simulation functions. This kind of functions have attracted much attention because they are useful to express a great family of contractivity conditions that were well known in the field of fixed point theory. Later on, Olgun et al. \cite{OBA} provided a new class of Picard operator on complete metric spaces using the concept of generalized $\mathcal{Z}-$contraction. In this exciting context, there are a lot of developments have been done in recent times \cite{NV,RKRM}.

In this manuscript, we use $\mathcal{Z}-$contractions to obtain existence and uniqueness fixed point results on $\theta-$metric spaces. Also, we introduce the concept of modified $\mathcal{Z}-$contractions there and go on to derive a fixed point result using them in the said spaces. Our main results are equipped with competent examples.

This document unfolds with preliminaries section, where we review some definitions, examples and notable
results that are involved in the sequel. The main results section comprises of some lemmas and fixed point results. These results extend, unify and generalize several results in the existing literature. Further we furnish some non-trivial examples to elicit the usability of the obtained theorems.

\section{Preliminaries}
At the outset, we dash some basic definitions and fundamental results off here. In the rest of this paper, $\mathbb{N}$ will stand for the set of all non-negative integers and $\mathbb{R}$ will denote the set of all real numbers.% Also, $X$ will represent a non-empty set.
 
Let $T: X\rightarrow X$ be a self-mapping. We say $x \in X$ is a fixed point of $T$ if $Tx = x$.
 
The following notion of simulation functions was first introduced by Khojasteh et al. in \cite{KSR}.
\begin{definition} \cite{KSR} \label{d1}
Let $\zeta: [0, \infty) \times [0, \infty) \rightarrow \mathbb{R}$ be a mapping. Then $\zeta$ is called a simulation function, if it satisfies:
\begin{enumerate}
\item [($\zeta1$)]
$\zeta(0,0)=0 $,
\item [($\zeta2$)]
$\zeta(t,s)< s-t$ for all $s,t>0 $,
\item [($\zeta3$)]
if $\{t_n\},\{s_n\}$ are sequences defined in $(0, \infty)$ such that $lim_{n \rightarrow \infty}t_n= lim_{n \rightarrow \infty}s_n>0$, then 
\[\limsup_{n \rightarrow \infty} \zeta(t_n, s_n)< 0.\]
\end{enumerate}
\end{definition}
The authors provided a wide range of examples of simulation functions to emphasize the promising applicability to the literature of fixed point theory. We list a few here.

\begin{example} \cite{KSR}
Let $\zeta_i: [0, \infty) \times [0, \infty), i=1,2$  be defined by:
\begin{enumerate}
\item
$\zeta_1(t, s)= \frac{s}{s+1}-t$ for all $t,s \in  [0, \infty).$
\item
$\zeta_2(t, s)= \eta(s)-t$ for all $t,s \in  [0, \infty)$, where $\eta: [0, \infty) \rightarrow [0, \infty)$ is an upper semi continuous mapping such that $\eta(t)< t$ for all $t>0$ and $\eta(0)=0$.
\item $\zeta_3(t, s)= s -\phi(s)-t$ for all $s, t\in [0,\infty)$ where, $\phi: [0,\infty) \rightarrow [0,\infty)$ is a continuous function such that $\phi(t) = 0 \Leftrightarrow t = 0$.
\end{enumerate}
\end{example}
The set of all simulation functions is denoted by $\mathcal{Z}.$

\begin{definition} \cite{KSR} \label{d2}
Let $T: X \rightarrow X$ be a self-mapping and $\zeta \in \mathcal{Z}$. Then $T$ is called a $\mathcal{Z}-$contraction with respect to $\zeta$, if 
\[ \zeta(d(Tx,Ty),d(x,y)) \geq 0\] holds for all $x,y \in X.$
\end{definition}
The Banach contraction is a perfect example of $\mathcal{Z}-$contraction. It satisfies the previous non-negativity restriction by taking $\zeta(t,s)= \lambda s- t$, where $\lambda \in [0,1),$ as the corresponding simulation function.

Despite the above examples, there are several other examples of simulation functions and $\mathcal{Z}-$contractions, which can be found on \cite{KSR}.
\begin{remark} [cf. \cite{KSR}] \label{r1}
It can be easily said from the definition of the simulation function that $\zeta(t,s)<0$ for all $t \geq s>0.$ So, if $T$ is a $\mathcal{Z}-$contraction with respect to $\zeta \in \mathcal{Z}$, then \[d(Tx,Ty)< d(x,y)\] whenever $x \neq y,$ for all $x,y \in X$.
This leads us to the conclusion that every $\mathcal{Z}-$contraction is contractive and hence continuous.
\end{remark}

For our purposes, we need to enunciate the ideas of $B$-actions and $\theta$-metrics here. In 2013, Khojasteh et al. \cite{KKR} proposed the notion of $\theta-$metric as a proper generalization of a metric.
\begin{definition} \cite{KKR} \label{d3}
Let $\theta : [0, \infty) \times [0, \infty) \rightarrow [0, \infty)$ be a continuous mapping with respect to both the variables. Let Im$(\theta)=\{\theta(s,t): s\geq 0, t\geq 0\}$. The mapping $\theta$ is called an $B$-action if and only if it satisfies the following conditions:
\begin{enumerate}
\item[(B1)]
$\theta(0,0)=0~~ and ~~\theta(s,t)= \theta(t,s)$  for all $ s,t \geq 0,$
\item[(B2)]
\begin{eqnarray*}
\theta(s,t) < \theta(u,v) & \Rightarrow & 
\begin{cases}
\text{either} ~~s <u, t \leq v\\
\text{or} ~~s \leq u, t < v,
\end{cases}
\end{eqnarray*}
\item[(B3)]
for each $r \in Im(\theta)$ and for each $s \in [0,r],$ there exists $t \in [0,r]$ such that $\theta
(t,s)=r,$
\item[(B4)]
$\theta(s,0) \leq s,$ for all $s> 0.$
\end{enumerate}
\end{definition}

\begin{example} \cite{KKR}
The subsequent examples illustrate the definition.
\begin{enumerate}
\item
$\theta_1(s,t)= \frac{ts}{1+ts}.$
\item
$\theta_2(s,t)= t+s+ \sqrt{ts}.$
\end{enumerate}
\end{example}
The set of all $B$-actions is denoted by $Y$.

The idea of $B-$action has been very much functional to formulate the notion of $\theta-$metric spaces \cite{KKR}. We here recall the definition of the said spaces.

\begin{definition} \cite{KKR} \label{d4}
Let $X$ be a non-empty set. A mapping $d_\theta: X \times X \rightarrow [0,\infty)$ is called a $\theta$-metric on $X$ with respect to $B-$action $\theta \in Y$ if $d_\theta$ satisfies the following:
\begin{enumerate}
\item[($\theta1$)]
$ d_\theta(x,y)=0$ if and only if $x=y,$
\item[($\theta2$)]
$ d_\theta(x,y)= d_\theta(y,x)$, for all $x,y \in X,$
\item[($\theta3$)]
$ d_\theta(x,y) \leq \theta(d_\theta(x,z),d_\theta(z,y)),$ for all $x,y,z \in X.$

\end{enumerate}
Then the pair $(X,d_\theta)$ is called a $\theta$-metric space.
\end{definition}
\begin{example} \cite{KKR}
Here we provide a non-trivial example of $\theta-$metric space.

Let $X= \{ a,b,c\}$ and $d_\theta: X \times X \rightarrow [0,\infty)$ is defined as:
\[d_\theta(x,y)= 5, d_\theta(y,z)= 12, d_\theta(z,x)=13, d_\theta(x,y)=d_\theta(y,x),\] \[ d_\theta(y,z)=d_\theta(z,y), d_\theta(z,x)=d_\theta(x,z), d_\theta(x,x)=d_\theta(y,y)=d_\theta(z,z)=0.\]
Taking $\theta(s,t)= \sqrt{s^2+t^2},$ the mapping $d_\theta$ forms a $\theta$-metric. And hence the pair $(X,d_\theta)$ is a $\theta$-metric space.
\end{example}

\begin{remark} [cf. \cite{KKR}]
If $(X,d_\theta)$ is a $\theta$-metric space and $\theta(s,t)=s+t,$ then $(X,d_\theta)$ is a metric space.
Also we mention that a metric space is included in the class of $\theta$-metric spaces if we consider the $\theta$-metric as $\theta(s,t)=s+t$. 
\end{remark}
For further terminology and derived results, we refer to \cite{KKR}.

\section{Main Results}
In this section, we prove some fixed point theorems for self-mappings via simulation functions owing to the concept of $\theta-$metric spaces and also we give illustrative examples. Before all else, we start with noting down following lemmas which will be crucial to our main results.
\begin{lemma} \label{l1}
If $(X,d_\theta)$ be any complete $\theta-$metric space and $T:X \rightarrow X$ be a $\mathcal{Z}-$contraction with respect to $\zeta \in \mathcal{Z}$, then $T$ is an asymptotically regular mapping at every $x \in X.$
\end{lemma}
\begin{proof}
Let $x \in X$ be any arbitrary element. Now, without loss of generality, we take $T^n x \neq T^{n+1}x,$ for all $n \in \mathbb{N}.$
Taking into account Remark \ref{r1}, we have, 
\begin{eqnarray*}
d_\theta(T^n x, T^{n+1}x) < d_\theta(T^n x, T^{n+1}x)
\end{eqnarray*}
for all $n \in \mathbb{N}$. 
So $\{d_\theta(T^n x, T^{n+1} x)\}$ is a decreasing sequence of non-negative reals. Thus there exists a $r\geq 0$ such that $lim_{n\rightarrow \infty} d_\theta(T^n x, T^{n+1} x)= r.$ Our claim is that $r=0.$
Since $T$ is a $\mathcal{Z}-$contraction with respect to $\zeta$, we have
\begin{eqnarray*}
0 & \leq & \limsup_{n \rightarrow \infty}\zeta(d_\theta(T^{n+1} x,T^n x), d_\theta(T^n x,T^{n-1} x))\\
& < & 0.
\end{eqnarray*}
This contradiction proves that $r=0$ and hence $\lim_{n \rightarrow \infty}d_\theta(T^n x, T^{n+1} x)=0$.
So $T$ is asymptotically regular mapping at every $x \in X.$
\end{proof}
\begin{lemma} \label{l2}
Let $(X,d_\theta)$ be any complete $\theta-$metric space and $T:X \rightarrow X$ be a $\mathcal{Z}-$contraction with respect to $\zeta \in \mathcal{Z}$. Then if $T$ has any fixed point in $X$, then it is unique.
\end{lemma}
\begin{proof}
Let $u \in X$ be any fixed point of $T$. We take $v \in X$ as another fixed point of $T$ with $u \neq v$. 
Therefore, $Tu=u$ and $Tv=v.$ Now by using $(B4)$ and $(\theta3),$ we obtain
\begin{eqnarray*}
d_\theta(u,v) & =& d_\theta(Tu,Tv)\\
& \leq & \theta(d_\theta(Tu,u),d_\theta(u,Tv))\\
& = & \theta(d_\theta(u,u),d_\theta(u,Tv))\\
& \leq & d_\theta(u,Tv)\\
& \leq & \theta(d_\theta(u,v),d_\theta(v,Tv))\\
& \leq & \theta(d_\theta(u,v),d_\theta(v,v))\\
& \leq & d_\theta(u,v).
\end{eqnarray*}
%0  (d(Tu; Tv); d(u; v)) = (d(u; v); d(u; v)):
In view of Remark \ref{r1}, above inequality yields a contradiction and hence proves result.
\end{proof}

The first main result of this article is the following one.
\begin{theorem}\label{t1}
Let $(X,d_\theta)$ be any complete $\theta-$metric space and $T:X \rightarrow X$ be a $\mathcal{Z}-$contraction with respect to $\zeta \in \mathcal{Z}$. Then $T$ has a unique fixed point $u$ in $X$ and for every $x_0 \in X,$ the Picard sequence $\{x_n\}$ converges to the fixed point of $T$.
\end{theorem}
\begin{proof}
Let $x_0$ be any arbitrary point and $\{x_n\}$ be the corresponding Picard sequence, i.e., $x_n=Tx_{n-1}$ for all $n \in \mathbb{N}$. We claim that the sequence $\{x_n\}$ is bounded. 

Reasoning by contradiction, we assume that, $\{x_n\}$ is unbounded. So, there exists a subsequence $\{x_{n_k}\}$ of $\{x_n\}$ such that $n_1=1$ and for each $k \in \mathbb{N}, n_{k+1}$ is minimum integer such that
\[d_\theta(x_{n_{k+1}},x_{n_k}) > 1\]
and 
\begin{eqnarray}
d_\theta(x_m,x_{n_k}) & \leq & 1
\end{eqnarray}
for $n_k \leq m \leq n_{k+1}-1.$
Now, using the triangle inequality $(\theta3)$ and $(3.1)$, we have
\begin{eqnarray}
1 & < & d_\theta(x_{n_{k+1}},x_{n_k}) \nonumber\\
& \leq & \theta(d_\theta(x_{n_{k+1}},x_{n_{k+1}-1}),d_\theta(x_{n_{k+1}-1},x_{n_k})) \nonumber\\
& \leq & \theta(d_\theta(x_{n_{k+1}},x_{n_{k+1}-1}),1).
\end{eqnarray}
Letting $k \rightarrow \infty$ on both sides of $(3.2)$ and then using Lemma \ref{l1} and $(B4)$, we deduce that,
 \[d_\theta(x_{n_{k+1}},x_{n_k}) \rightarrow 1.\]
On the other hand, using $(\theta3)$ and $(3.1)$, we derive that 
\begin{eqnarray*}
1 & < & d_\theta(x_{n_{k+1}},x_{n_k}) \\
& \leq & d_\theta(x_{n_{k+1}-1},x_{n_k-1})\\
& \leq & \theta(d_\theta(x_{n_{k+1}-1},x_{n_k}),d_\theta(x_{n_k},x_{n_k-1}))\\
& \leq & \theta(1,d_\theta(x_{n_k},x_{n_k-1})).
\end{eqnarray*}
So, as $k \rightarrow \infty$, we get, \[d_\theta(x_{n_{k+1}-1},x_{n_k-1}) \rightarrow 1.\]
Since $T$ is a $\mathcal{Z}-$contraction with respect to $\zeta \in \mathcal{Z},$ we derive that
\begin{eqnarray*}
0 & \leq & lim sup_{k \rightarrow \infty} \zeta(d_\theta(Tx_{n_{k+1}-1},Tx_{n_k-1}),d_\theta(x_{n_{k+1}-1},x_{n_k-1}))\\
& = & lim sup_{k \rightarrow \infty} \zeta(d_\theta(x_{n_{k+1}},x_{n_k}),d_\theta(x_{n_{k+1}-1},x_{n_k-1}))\\
& < & 0,
\end{eqnarray*}
and we arrive at a contradiction.
So, the Picard sequence $\{x_n\}$ is bounded.

Now we will show that  $\{x_n\}$ is Cauchy. For this, let \[C_n= sup\{d_\theta(x_i,x_j): i,j \geq n\}.\]
Note that $\{C_n\}$ is a decreasing sequence of non-negative reals. Thus there exists a $C\geq 0$ such that $lim_{n\rightarrow \infty} C_n=C.$ Our claim is that $C=0.$ Let us suppose that $C>0.$
Then by the definition of $C_n$, for every $k \in \mathbb{N},$ there exists $n_k, m_k$ such that $m_k> n_k \geq k$ and
\begin{eqnarray*}
C_k - \frac{1}{k} & < d_\theta(x_{m_k},x_{n_k}) & \leq C_k. 
\end{eqnarray*}
Letting $k \rightarrow \infty$ in the above inequality, we get
\begin{eqnarray*}
lim_{k \rightarrow \infty}d_\theta(x_{m_k},x_{n_k})=C.
\end{eqnarray*}
Now,
\begin{eqnarray*}
d_\theta(x_{m_k},x_{n_k}) & \leq & d_\theta(x_{m_k-1},x_{n_k-1})\\
& \leq &  \theta(d_\theta(x_{m_k-1},x_{m_k}),d_\theta(x_{m_k},x_{n_k-1}))\\
& \leq &  \theta(d_\theta(x_{m_k-1},x_{m_k}),\theta(d_\theta(x_{m_k},x_{n_k}),d_\theta(x_{n_k},x_{n_k-1}))).
\end{eqnarray*}
Letting  $k \rightarrow \infty$ in the previous inequality and applying $(B4)$, we derive
\begin{eqnarray}
C & \leq & \lim_{k \rightarrow \infty}d_\theta(x_{m_k-1},x_{n_k-1}) \nonumber \\
& \leq & \theta(0,\theta(d_\theta(x_{m_k},x_{n_k}),d_\theta(x_{n_k},x_{n_k-1}))) \nonumber \\
& \leq & \theta(d_\theta(x_{m_k},x_{n_k}),d_\theta(x_{n_k},x_{n_k-1})).
\end{eqnarray}
Again taking limit as $k \rightarrow \infty$ in $(3.3)$ and using $(B4)$, we get
\begin{eqnarray*}
C & \leq &\lim_{k \rightarrow \infty}d_\theta(x_{m_k-1},x_{n_k-1}) \\
& \leq & \theta(0,C)\\
& \leq & C.
\end{eqnarray*}
As a consequence, 
\[ lim_{k \rightarrow \infty} d_\theta(x_{m_k-1},x_{n_k-1})=C.\]
Now since $T$ is a $\mathcal{Z}-$contraction with respect to $\zeta \in \mathcal{Z},$ we derive that
\begin{eqnarray*}
0 & \leq & lim sup_{k \rightarrow \infty} \zeta(d_\theta(x_{m_k-1},x_{n_k-1}),d_\theta(x_{m_k},x_{n_k}))\\
& < & 0,
\end{eqnarray*}
which is a contradiction. Consequently, $\{x_n\}$ is Cauchy.

Since $(X,d_\theta)$ is complete, there exists some $z \in X$ such that $lim_{n \rightarrow \infty} x_n=z.$

Now we show that $z$ is a fixed point of $T.$ Conversely suppose, $Tz \neq z.$ Then $d_\theta(z,Tz)>0.$
Again,
\begin{eqnarray*}
0 & \leq & lim sup_{n \rightarrow \infty} \zeta(d_\theta(Tx_n,Tz),d_\theta(x_n,z))\\
& \leq & lim sup_{n \rightarrow \infty} [d_\theta(x_n,z)-d_\theta(x_{n+1},Tz)]\\
& = & -d_\theta(z,Tz).
\end{eqnarray*}
This contradiction proves that $d_\theta(z,Tz)=0,$ and hence, $Tz=z.$ So we can conclude that $z$ is a fixed point of $T.$ Uniqueness is guaranteed from Lemma \ref{l2}. 
\end{proof}
Now we validate our fixed point result by the following examples.
\begin{example}
Let $X= [0,1]$ be endowed with the Euclidean metric $d_\theta(x,y)=|x-y|$. Also we take $\theta(s,t)=s+t+st.$

We define a mapping $T: X \rightarrow X$ by $Tx = \frac{x}{a}+b,$ where $a > 1, x \in X$ and $b+\frac{1}{a} <1$.

So we have,
\begin{eqnarray*}
d_\theta(Tx,Ty) & = & |Tx-Ty|\\
&=& |\frac{x}{a}+b-\frac{y}{a}-b| \\
&=& \frac{1}{a}|x-y|.
\end{eqnarray*}
We claim that $T$ is a $\mathcal{Z}-$contraction with respect to the simulation function $\zeta(t,s)=\lambda s - t$, where $\lambda > \frac{1}{a},$  for all $t,s \in [0, \infty)$.

So we have,
\begin{eqnarray*}
\zeta(d_\theta(Tx,Ty),d_\theta(x,y)) & = & \lambda d_\theta(x,y)- d_\theta(Tx,Ty)\\
& = & \lambda |x-y|- \frac{1}{a}|x-y|\\
& = & (\lambda - \frac{1}{a}) |x-y|\\
& \geq & 0.
\end{eqnarray*}
Taking into account Theorem \ref{t1} we get, $T$ has a unique fixed point and it is $u=\frac{ab}{a-1}.$

Since $b+\frac{1}{a} <1,$ it is ensured that $u \in X.$
\end{example}

\begin{example}
Let $X= [0,1]$ be endowed with the Euclidean metric and $\theta(s,t)=s+t+st.$

We define a mapping $T: X \rightarrow X$ by $Tx = \frac{1}{1+x}, $ $ x \in X.$

Our claim is that $T$ is a $\mathcal{Z}-$contraction with respect to the simulation function $\zeta(t,s)=\frac{s}{s+1}-t,$ for all $t,s \in [0, \infty)$.
 
So we have,
\begin{eqnarray*}
\zeta(d_\theta(Tx,Ty),d_\theta(x,y))& = & \frac{d_\theta(x,y)}{d_\theta(x,y)+1}-d_\theta(Tx,Ty)\\
& = & \frac{|x-y|}{|x-y|+1}- |\frac{1}{x+1}-\frac{1}{y+1}|\\
& = & \frac{|x-y|}{|x-y|+1}- \frac{|x-y|}{|x+1||y+1|}\\
& = & |x-y|(\frac{1}{|x-y|+1}- \frac{1}{|x+1||y+1|})\\
& \geq & 0.
\end{eqnarray*}
Hence applying Theorem \ref{t1}, $T$ has a unique fixed point and it is $u=\frac{\sqrt 5 -1}{2} \in X.$

\end{example}
Here we introduce the new class of modified $\mathcal{Z}-$contractions.
\begin{definition} \label{d5}
Let $T: X \rightarrow X$ be a mapping and $\zeta \in \mathcal{Z}$. Then $T$ is called a modified $\mathcal{Z}-$contraction with respect to $\zeta$, if it satisfies:
\[ \zeta(d_\theta(Tx,Ty),M(x,y)) \geq 0\] for all $x,y \in X,$
where, 
\begin{eqnarray*}
M(x,y) & = & max\{d_\theta(x,y),d_\theta(x,Tx),d_\theta(y,Ty)\}.
\end{eqnarray*}
\begin{example}
Let $X= [0,1]$ be endowed with the Euclidean metric and $\theta(s,t)=s+t+st.$
We define a mapping $T: X \rightarrow X$ by
\begin{eqnarray*}
Tx & = & 
\begin{cases}
\frac{1}{7},~~ x \in S_1=[0,\frac{1}{2}), \\
\frac{2}{7},~~  x \in S_2=[\frac{1}{2},1].
\end{cases}
\end{eqnarray*}
Then $T$ is a modified $\mathcal{Z}-$contraction with respect to the simulation function $\zeta(t,s)= \frac{7}{8}s - t.$
\end{example}
\end{definition}
Now we deliver one of our main results related to modified $\mathcal{Z}-$contraction on the context of $\theta-$metric spaces. This theorem assures us about the existence and uniqueness of the fixed point of a modified $\mathcal{Z}-$contraction. The subsequent lemma forms the basis for our result.
\begin{lemma} \label{l3}
Let $(X,d_\theta)$ be any complete $\theta-$metric space and $T:X \rightarrow X$ be a modified $\mathcal{Z}-$contraction with respect to $\zeta \in \mathcal{Z}$. Then if $T$ has any fixed point in $X$, it is unique.
\end{lemma}
\begin{proof}
Let $u \in X$ be any fixed point of $T$.Suppose $v \in X$ be another fixed point of $T$ with $u \neq v$. 
This means that $Tu=u$ and $Tv=v.$

From Definition \ref{d5} and using the previous fact, we observe that
\begin{eqnarray*}
 M(u,v) & = & max\{d_\theta(u,v),d_\theta(u,Tu),d_\theta(v,Tv)\}\\
  & = & max\{d_\theta(u,v),d_\theta(u,u),d_\theta(v,v)\}\\
  & = & d_\theta(u,v).
\end{eqnarray*}
Since $T$ is a modified $\mathcal{Z}-$contraction with respect to $\zeta \in \mathcal{Z},$ we attain that
\begin{eqnarray*}
 0 & \leq & \zeta(d_\theta(Tu,Tv),M(u,v))\\
& = & \zeta(d_\theta(Tu,Tv),d_\theta(u,v))\\
& = & \zeta(d_\theta(u,v),d_\theta(u,v)).\\
\end{eqnarray*} 
Considering Lemma \ref{r1}, above inequality reaches a contradiction and hence proves result.

\end{proof}
Now, we are ready to state our another main result here.
\begin{theorem}\label{t2}
Let $(X,d_\theta)$ be any complete $\theta-$metric space and $T:X \rightarrow X$ is a modified $\mathcal{Z}-$contraction with respect to $\zeta \in \mathcal{Z}$. Then $T$ has a unique fixed point $u$ in $X$ and for every $x_0 \in X,$ the Picard sequence $\{x_n\}$ converges to the fixed point of $T$.
\end{theorem}
\begin{proof}
Let $(X, d_\theta)$ be a $\theta-$metric space and $T:X \rightarrow X$ be a modified $\mathcal{Z}-$contraction with respect to $\zeta \in \mathcal{Z}$. 

Let $x_0$ be any arbitrary point and $\{x_n\}$ be the respective Picard sequence, i.e., $x_n=Tx_{n-1}$ for all $n \in \mathbb{N}$. Now we suppose that $d_\theta(x_n,x_{n+1})> 0$ for all $n \in \mathbb{N}.$ Otherwise if there exists $n_p \in \mathbb{N}$ such that $x_{n_p}= x_{n_p+1}$, then $x_{n_p}$ is a fixed point of $T$ and we are done.

Next we define $d_{\theta}^{n}=d_\theta(x_n, x_{n+1})$.
Then, since, 
\begin{eqnarray*}
M(x_n,x_{n+1}) & = & max\{d_\theta(x_n,x_{n-1}),d_\theta(x_n,x_{n+1}),d_\theta(x_{n-1},x_n)\}\\
& = & max\{d_{\theta}^{n},d_{\theta}^{n-1}\}.
\end{eqnarray*}
Since $\{d_{\theta}^{n}\}$ is a decreasing sequence of reals, $d_{\theta}^{n} < d_{\theta}^{n-1}$ for all $n \in \mathbb{N}.$
So we get, 
\begin{eqnarray*}
0 & \leq & \zeta(d_\theta(Tx_n,Tx_{n-1}),M(x_n,Tx_{n-1}))\\
& = & \leq \zeta(d_{\theta}^{n},max\{d_{\theta}^{n},d_{\theta}^{n-1}\})\\
& = & \leq \zeta(d_{\theta}^{n},d_{\theta}^{n-1}).
%\zeta(_{\theta}^{n}, d_{\theta}^{n-1}).
\end{eqnarray*}
Now $\{d_{\theta}^{n}\}$ is a decreasing sequence of non-negative real numbers and hence is convergent. Let $lim_{n \rightarrow \infty}d_{\theta}^{n}= r.$
If $r> 0,$ we have, 
\begin{eqnarray*}
0 & \leq & limsup_{n \rightarrow \infty} \zeta(d_{\theta}^{n},d_{\theta}^{n-1})\\
& < & 0. 
\end{eqnarray*}
We arrive at a contradiction, so $r =0.$ Therefore $T$ has a fixed point.
%Hence, $ \theta(d_\theta(x_n, x_{n+1}),d_\theta(x_n, x_{n+1}))=0.$ And so, $d_\theta(x_n, x_{n+1})=0 \Rightarrow d_\theta(x_n, Tx_n)=0.$

We claim that the sequence $\{x_n\}$ is bounded. Reasoning by contradiction, we assume that, $\{x_n\}$ is unbounded. So, there exists a subsequence $\{x_{n_k}\}$ of $\{x_n\}$ such that $n_1=1$ and for each $k \in \mathbb{N}, n_{k+1}$ is minimum integer such that
\[d_\theta(x_{n_{k+1}},x_{n_k}) > 1\]
and \[d_\theta(x_m,x_{n_k}) \leq 1 ~~for ~~n_k \leq m \leq n_{k+1}-1.\]
Now, using the triangle inequality, we have
\begin{eqnarray}
1 & < & d_\theta(x_{n_{k+1}},x_{n_k}) \nonumber \\
& \leq & \theta(d_\theta(x_{n_{k+1}},x_{n_{k+1}-1}),d_\theta(x_{n_{k+1}-1},x_{n_k})) \nonumber \\
& \leq & \theta(d_\theta(x_{n_{k+1}},x_{n_{k+1}-1}),1).
\end{eqnarray}
By taking the limit as $k \rightarrow \infty$ on both sides of $(3.4)$ and using $(B4)$, we infer that,
 \[d_\theta(x_{n_{k+1}},x_{n_k}) \rightarrow 1.\]
Also, we have 
\begin{eqnarray*}
1 & < & d_\theta(x_{n_{k+1}},x_{n_k}) \\
& \leq & M(x_{n_{k+1}-1},x_{n_k-1})\\
& \leq & max\{d_\theta(x_{n_{k+1}-1},x_{n_k-1}),d_\theta(x_{n_{k+1}-1},x_{n_{k+1}}),d_\theta(x_{n_k-1},x_{n_k})\}\\
& \leq & max\{\theta(d_\theta(x_{n_{k+1}-1},x_{n_k}),d_\theta(x_{n_k},x_{n_k-1})),d_\theta(x_{n_{k+1}-1},x_{n_{k+1}}),d_\theta(x_{n_k-1},x_{n_k})\}\\
& \leq & max\{\theta(1,d_\theta(x_{n_k},x_{n_k-1})),d_\theta(x_{n_{k+1}-1},x_{n_{k+1}}),d_\theta(x_{n_k-1},x_{n_k})\}.
\end{eqnarray*}
As $k \rightarrow \infty$, we derive, $ 1 \leq M(x_{n_{k+1}-1},x_{n_k-1}) \leq 1.$

So, \[lim_{k \rightarrow \infty}M(x_{n_{k+1}-1},x_{n_k-1})=1.\]
As $T$ is a modified $\mathcal{Z}-$contraction with respect to $\zeta \in \mathcal{Z},$ we obtain 
\begin{eqnarray*}
0 & \leq & lim sup_{k \rightarrow \infty} \zeta(d_\theta(x_{n_{k+1}},x_{n_k}),M(x_{n_{k+1}},x_{n_k}))\\
& < & 0.
\end{eqnarray*}
This leads to a contradiction and hence $\{x_n\}$ is bounded.

Now we will show that  $\{x_n\}$ is Cauchy. For this, we consider the real sequence \[C_n= sup\{d_\theta(x_i,x_j: i,j \geq n)\}.\]
Note that $\{C_n\}$ is a decreasing sequence of non-negative reals. Thus there exists $C\geq 0$ such that $lim_{n\rightarrow \infty} C_n=C.$ Our claim is that $C=0.$ Let us suppose that $C>0.$
Then by the definition of $C_n$, for every $k \in \mathbb{N},$ there exists $n_k, m_k$ such that $m_k> n_k \geq k$ and
\begin{eqnarray*}
C_k - \frac{1}{k} & < d_\theta(x_{m_k},x_{n_k}) & \leq C_k. 
\end{eqnarray*}
Letting $k \rightarrow \infty$ in the above inequality, we get
\begin{eqnarray*}
lim_{k \rightarrow \infty}d_\theta(x_{m_k},x_{n_k})=C,
\end{eqnarray*}
and,
\begin{eqnarray*}
lim_{k \rightarrow \infty}d_\theta(x_{m_k-1},x_{n_k-1})=C.
\end{eqnarray*}
Now,
\begin{eqnarray*}
d_\theta(x_{m_k},x_{n_k}) & \leq & M(x_{m_k-1},x_{n_k-1})\\
& = &  max\{d_\theta(x_{m_k-1},x_{n_k-1}),d_\theta(x_{m_k-1},x_{m_k}),d_\theta(x_{n_k-1},x_{n_k})\},\\
& = &  max\{d_\theta(x_{m_k-1},x_{n_k-1}),d_\theta(x_{m_k-1},x_{m_k}),d_\theta(x_{n_k-1},x_{n_k})\}.
\end{eqnarray*}
Consequently, taking $k \rightarrow \infty$, we get, 
\[ lim_{k \rightarrow \infty} M(x_{m_k-1},x_{n_k-1})=C.\]
Now since $T$ is a modified $\mathcal{Z}-$contraction with respect to $\zeta \in \mathcal{Z},$ we derive that
\begin{eqnarray*}
0 & \leq & lim sup_{k \rightarrow \infty} \zeta(d_\theta(x_{m_k},x_{n_k}),M(x_{m_k-1},x_{n_k-1}))\\
& < & 0,
\end{eqnarray*}
which is a contradiction. As a result, $C=0$ and $\{x_n\}$ is Cauchy.

Since $(X,d_\theta)$ is complete, there exists some $z \in X$ such that $lim_{n \rightarrow \infty} x_n=z.$ Now we show that $z$ is a fixed point of $T.$
Suppose, on the contrary, $Tz \neq z.$ Then $d_\theta(z,Tz)>0.$
Now we employ Definition \ref{d5} and use Remark \ref{r1} to get,
\begin{eqnarray*}
0 & \leq & lim sup_{n \rightarrow \infty} \zeta(d_\theta(Tx_n,Tz),M(x_n,z))\\
& \leq & lim sup_{n \rightarrow \infty} [M(x_n,z)-d_\theta(x_{n+1},Tz)]\\
& = & -d_\theta(z,Tz).
\end{eqnarray*}
This contradiction attests that $d_\theta(z,Tz)=0,$ and so, $Tz=z.$ Thus $z$ is a fixed point of $T.$ Uniqueness is guaranteed from Lemma \ref{l3}.
\end{proof}
As an application of our earlier result, we furnish the next example which illustrates Theorem \ref{t2}.
\begin{example}
Let $X= [0,1]$ be equipped with the usual Euclidean metric and $\theta(s,t)=s+t+st.$
We define a mapping $T: X \rightarrow X$ by
\begin{eqnarray*}
Tx & = & 
\begin{cases}
\frac{2}{9},~~ x \in S_1=[0,\frac{1}{2}), \\
\frac{1}{9},~~  x \in S_2=[\frac{1}{2},1].
\end{cases}
\end{eqnarray*}
We argue that $T$ is a modified $\mathcal{Z}-$contraction with respect to the simulation function $\zeta(t,s)= \frac{1}{2}s - t.$

Here we have, $0 \leq d_\theta(Tx,Ty) \leq \frac{1}{9}$ for all $x,y \in X.$

Now, if both $x, y \in S_1$ or $S_2$, then $d_\theta(Tx,Ty)=0$ and we are done.

Otherwise, let $x \in S_1$ and $y \in S_2$.

We get $0 < d_\theta(x,y) \leq 1$. Also, $\frac{2}{9} \leq d_\theta(x,Tx) \leq \frac{5}{9}$ and $\frac{7}{18} \leq d_\theta(y,Ty) \leq \frac{8}{9}.$
Therefore $M(x,y) \geq \frac{7}{18}.$
From the calculation, it is clear that \[d_\theta(Tx,Ty) \leq \frac{1}{2}M(x,y).\]
So we have, \[ \zeta(d_\theta(Tx,Ty),M(x,y))=  \frac{1}{2}M(x,y)- d_\theta(Tx,Ty) \geq 0\] for all $x,y \in X.$

As a consequence $T$ is a modified $\mathcal{Z}-$contraction. Taking into account Theorem \ref{t2}, we can say that $T$ has a unique fixed point. Here $u= \frac{2}{9}$ is that required fixed point.
\end{example}

\begin{Acknowledgement}
The first named author would like to convey his cordial thanks to DST-INSPIRE, New Delhi, India for their financial support under INSPIRE fellowship scheme.
%We gratefully acknowledge the critical review by the learned referee which reasonably improves the manuscript.
\end{Acknowledgement}


\begin{thebibliography}{10}

\bibitem{ba}
S.~Banach, \textit{Sur les op\'erations dans les ensembles abstraits et leurs
  applications aux \'equations int\'egralesi}, Fund. Math. \textbf{3} (1922),
  133--181.

\bibitem{Ch}
S.~Chatterjea, \textit{Fixed-point theorems}, C. R. Acad. Bulgare Sci.
  \textbf{25} (1972), 727--730.

\bibitem{DD}
P.~Das and L.~Dey, \textit{Fixed point of contractive mappings in generalized
  metric spaces}, Math. Slovaca \textbf{59} (2009), 499--504.

\bibitem{DM}
L.~Dey and S.~Mondal, \textit{Best proximity point of {F}-contraction in
  complete metric space}, Bull. Allahabad Math. Soc. \textbf{30} (2015),
  173--189.

\bibitem{De}
L.~Dey, S.~Mondal and A.~Chanda, \textit{Common fixed point and best proximity
  point theorem for ${C^*}$-algebra-valued metric spaces}, 2015, submitted to
  the Fixed Point Theory.

\bibitem{GRRS}
R.~George, S.~Radenovi\'{c}, K.~P. Reshma and S.~Shukla, \textit{Rectangular
  b-metric space and contraction principles}, J. Nonlinear Sci. Appl.
  \textbf{8} (2015), 1005--1013.

\bibitem{Rk1}
R.~Kannan, \textit{Some results on fixed points- {II}}, Amer. Math. Monthly
  \textbf{76} (1969), 405--408.

\bibitem{KKR}
F.~Khojasteh, E.~Karapinar and S.~Radenovi\'{c}, \textit{$\theta-$metric space:
  A generalization}, Math. Probl. Eng. \textbf{2013}, article ID 504609.

\bibitem{KSR}
F.~Khojasteh, S.~Shukla and S.~Radenovi\'{c}, \textit{A new approach to the
  study of fixed point theory for simulation functions}, Filomat \textbf{29}
  (2015), 1089--1194.

\bibitem{NV}
A.~Nastasi and P.~Vetro, \textit{Fixed point results on metric and partial
  metrric spaces via simulations functions}, J. Nonlinear Sci. Appl. \textbf{8}
  (2015), 1059--1069.

\bibitem{OBA}
M.~Olgun, O.~Bicer and T.~Alyildiz, \textit{A new aspect to picard operators
  with simulation functions}, Turkish J. Math. \textbf{40} (2016), 832--837.

\bibitem{RKRM}
A.~Roldan, E.~Karapinar, C.~Roldan and J.~Martinez-Moreno, \textit{Coincidence
  point theorems on metric spaces via simulation functions}, J. Comput. Appl.
  Math.  (2015), 345--355.

\end{thebibliography}
\end{document}